\documentclass[12pt]{amsart}

\usepackage{enumerate}
\usepackage{amsfonts}
\usepackage{amsthm}
\usepackage{amsmath}
\usepackage{amssymb}
\usepackage{t1enc}
\usepackage{fullpage}
\usepackage{dsfont}
\usepackage{graphicx}
\usepackage[a4paper,left=2.5cm,right=2.5cm, footskip=30pt]{geometry}

\newtheorem{theorem}{Theorem}[section]
\newtheorem{prop}{Proposition}[section]
\newtheorem{lemma}{Lemma}[section]
\newtheorem{defi}{Definition}[section]

\newtheorem{megj}{Remark}[section]

\begin{document}

\begin{abstract}
During the last few decades E. S. Thomas, S. J. Agronsky, J. G. Ceder, and T. L. Pearson gave an equivalent definition of the real Baire class 1 functions by characterizing their graph. In this paper, using their results, we consider the following problem: let $T$ be a given subset of $[0,1]\times\mathbb{R}$. When can we find a function $f:[0,1]\rightarrow\mathbb{R}$ such that the accumulation points of its graph are exactly the points of $T$? We show that if such a function exists, we can choose it to be a Baire-2 function. We characterize the accumulation sets of bounded and not necessarily bounded functions separately. We also examine the similar question in the case of Baire-1 functions.
\end{abstract}

\title{Accumulation Points of Graphs of Baire-1 and Baire-2 Functions}
\author{Bal\'azs Maga}
\date{}
\maketitle

\section{Introduction}

In the last sixty years, certain classes of real functions have been characterized with a description of their graphs. In the case of Baire-1 functions it is worth mentioning the article of E. S. Thomas and the article of Agronsky, Ceder, and Pearson (see \cite{B1} and \cite{B2}): in the former one an equivalent definition of bounded Baire-1 functions was given, in the latter this result was generalized for the not necessarily bounded case. In this paper we also investigate a property of graphs of Baire-1 and Baire-2 functions. The problem is the following: if $T$ is a given subset of $[0,1]\times\mathbb{R}$, when does there exist a Baire-1 or Baire-2 function $f:[0,1]\rightarrow\mathbb{R}$ such that the accumulation points of its graph are exactly the points of $T$? 

\medskip

\noindent We answer these questions in two steps in both cases. It is easier to understand the theorems and the proofs if we also require $f$ to be bounded, thus we start with this case. 

\section{Notation}

Throughout this paper we use the following notation: the graph of the real function $f$ is denoted by $G$. Analogously, the graph of $f_0$ is $G_0$. If $f$ is a real function, the set of accumulation points of $G$ is $L_f$. The vertical line given by the equation $x=r$ is denoted by $v_r$. If $H$ is a set of $\mathbb{R}^2$, and $r$ is a real number, the intersection of $v_r$ and $H$ is denoted by $H(r)$. For simplicity, if $(r,y)\in{H}$, we say that $y\in{H_r}$. The open ball with center $r$ and radius $\varepsilon$ is $B(r,\varepsilon)$. We use this notation for one-dimensional neighborhoods in $\mathbb{R}$, and also for two-dimensional neighborhoods in $\mathbb{R}^2$. We clarify this ambiguity by making clear if the center is a point of $\mathbb{R}$ or of $\mathbb{R}^2$. The interval $[0,1]$ is denoted by $I$. The cardinality of a set $H$ is $\#(H)$. The diameter of a set $H$ is $\text{diam}(H)$. Finally, if a set $A\subseteq{I}$ is the subset of the domain of $f$, and $a\in{A}$, sometimes we refer to the point $(a,f(a))$ as a point of $G$ above $A$.

\section{Preliminary Results}

In the introduction we have already mentioned the result of Agronsky, Ceder, and Pearson. This theorem will be a very useful tool for us, so it is appropriate to recall it. We need the following definition:

\begin{defi} An open set $S\subseteq\mathbb{R}^2$ is an open strip if for every $r\in{R}$ the set $S(r)$ is an open interval. \end{defi}

\noindent In [2, Theorem 2.2] a characterization of Baire-1 functions was given by using this definition:

\begin{prop} Let $f:I\rightarrow\mathbb{R}$ be a function. It is Baire-1 if and only if there is a sequence $(S_n)$ of open strips such that $\cap_{n=1}^{\infty}S_n=G$. \end{prop}

\noindent As we will see, this theorem is a truly useful tool if our goal is to show that a certain function is Baire-1. Besides that we will also apply the following lemma, which handles a variant of our original problem.

\begin{lemma} For a given closed set $T\subseteq{I\times\mathbb{R}}$, there exists a countable set $A\subseteq{I}$ such that there is a function $f:A\rightarrow\mathbb{R}$ satisfying $L_f=T$. \end{lemma}

\begin{proof} Let $T_i=(I\times[-i,i])\cap{T}$ for all $i\in\mathbb{N}$. Then every $T_i$ is compact. Let us consider an open ball of radius one around each point of $T_1$. These open balls cover $T_1$, hence it is possible to choose a finite covering. Let us take a point in each chosen open ball such that the $x$ coordinates of these points are pairwise different. Let us denote the set of these points by $H_1$, and the set of their $x$ coordinates by $A_1$.

Now, similarly, let us consider open balls with radius $\frac{1}{2}$ around each point of $T_2$ and choose a finite covering, then finally take points in these chosen neighborhoods and define $H_2$ and $A_2$ analogously. We can continue this procedure by induction: in the $n^{\textsl{th}}$ step we consider the $\frac{1}{n}$-neighborhoods of the points of $T_n$, and we define the finite sets $H_n$ and $A_n$ using these open balls.

Let $A=\cup_{n=1}^{\infty}{A_n}$ and $H=\cup_{n=1}^{\infty}{H_n}$. These are countable sets. Let $f$ be the function that assigns to every $x\in{A}$ the $y$ coordinate of the chosen point above $x$. Then this point of the graph is clearly a point of $H$. We would like to prove that $L_f=T$ for this function $f$. We do this by verifying two containments.

\begin{enumerate}

\item $T\subseteq{L_f}$. Let us consider any point $P$ of $T$. By definition, $P\in{T_k}$ for a suitable $k$ positive integer. Thus for every $n$ larger than $k$ there exists a point $x_n\in{A_n}$ such that the distance of $(x_n,f(x_n))$ and $P$ does not exceed $\frac{1}{n}$. Therefore, there exists a sequence of distinct points in $G$ that converges to $P$, hence $T\subseteq{L_f}$.

\item $L_f\subseteq{T}$. Let us consider any point $P$ of $L_f$. Since it is an accumulation point of $G$, there exists a sequence $(p_n)$ in $G$ converging to $P$ and containing each of its terms only once. Now if $k$ is given, for sufficiently large $n$ the point $p_n$ is in $H_m$ with $m\geq{k}$. It means that the distance of $p_n$ and $T$ does not exceed $\frac{1}{k}$. Thus there are points of $T$ arbitrarily close to the sequence $(p_n)$. Therefore, the limit of $(p_n)$ is in $T$, since $T$ is closed. Hence $P\in{T}$ and $L_f\subseteq{T}$. \end{enumerate} \end{proof}

\begin{megj} The above proof shows that there are only finitely many points of the graph $G$ that are more than $\varepsilon$ apart from $T$ for a given $\varepsilon>0$. Later we will use this slightly stronger result. \end{megj}

\section{Functions of Baire Class 2}

As we have promised, we consider the bounded case first. It is obvious that if $L_f=T$, then $T$ must be a compact set, being bounded and closed. There is another condition needed: $T(x)$ is never empty for $x\in{I}$. Indeed, if $(x_n)$ is a sequence that converges to $x$, ($x_n\neq{x}$), the sequence formed by the points $(x_n,f(x_n))$ is a bounded sequence in $\mathbb{R}^2$, and its limit is in $T$, thus $T(x)\neq\emptyset$.

We point out that until this point we have not used the Baire-2 property of the function $f$. Despite that, as we will see, these conditions are also sufficient:

\begin{theorem} Suppose $T\subseteq{I}\times\mathbb{R}$. There exists a bounded Baire-2 function $f:I\rightarrow\mathbb{R}$ such that $L_f=T$ if and only if 
\begin{itemize}
\item $T$ is compact,
\item $T(x)$ is nonempty for $x\in{I}$. 
\end{itemize}
\end{theorem}

\begin{proof} Before beginning the formal proof, we give a short sketch. First, we construct a function $f_0$ such that $f_0(x)\in{T(x)}$ for every $x\in{I}$. After this step, we apply Proposition 3.1 to prove that $f_0$ is a Baire-1 function. Finally, we use Lemma 3.1 to modify $f_0$ on a countable set $A$ to obtain a bounded Baire-2 function $f$ such that $L_f=T$.

Put $f_0(x)=\max(T(x))$ for every $x\in{I}$. Since $T(x)$ is nonempty, this definition makes sense. The function $f_0$ is Baire-1; this is a well-known fact since $f_0$ is upper semicontinuous and every upper semicontinuous function is Baire-1. Nevertheless, it is useful to find a direct proof which uses Proposition 3.1 to understand better how this theorem works.

We define a nested sequence of open strips, $(S_n)$. First, we  construct a subset $S_n'$ of $S_n$, that is the union of certain neighborhoods of points of $G_0$. Let the radius of such an open ball be $\varepsilon_{x,n}$, where $\varepsilon_{x,n}$ satisfies the following three conditions: $\varepsilon_{x,n}\leq{\frac{1}{n}}$ and $\varepsilon_{x,n}\leq\varepsilon_{x,n-1}$ for every $n\geq{2}$. It is obviously possible. Moreover, we have a bit more complicated so-called overlapping condition related to the projection of the open balls $B((x,f_0(x)),\varepsilon_{x,n})$ to the $x$-axis. Specifically:

$$ \forall x\in{I}, \forall n\in\mathbb{N}, \forall r\in\mathbb{R}, r\in{B(x,\varepsilon_{x,n})}\text{ we have } f_0(r)-f_0(x)<\frac{1}{n} \text{ .}$$

\noindent Such $\varepsilon_{x,n}$ can be chosen. If not, then there is a sequence $(x_k)$ that converges to $x$ and $f_0(x_k)\geq{f_0(x)+\frac{1}{n}}$ for every $k$. In this case $(f_0(x_k))$ is a bounded sequence, so it has a convergent subsequence. As a consequence, the sequence $(x_k,f_0(x_k))$ has a limit point in the plane whose first coordinate is $x$, and whose second coordinate is larger than $f_0(x)=\max(T(x))$ by at least $\frac{1}{n}$. Since $T$ is closed, it is a contradiction. 

\noindent Thus for every $n\in\mathbb{N}$ and $x\in{I}$, we can choose some $\varepsilon_{x,n}$ satisfying all three of our conditions. By taking the union of the neighborhoods $B((x,f_0(x)),\varepsilon_{x,n})$, we obtain an open set $S_n'$ containing $G_0$ for every $n$. Also $S_n'\subseteq{S_{n-1}'}$ for every $n\geq{2}$, since $S_n'$ is the union of open balls with the same centers and smaller radii. However, it is not sufficient for us: our aim is to construct open strips. But this problem can be solved easily. Specifically, there is a simple way to extend an arbitrary open set $H'$ to an open strip $H$: for every $x$, let $H(x)=(\inf(H'(x)),\sup(H'(x))$. Figure 1 demonstrates such an extension, in a case where $H'$ is the union of a few open disks: $H$ is the open set bounded by the dashed lines. It is plain to see that the set $H$ made this way is an open strip which contains $H'$. We also use this method to construct $S_n(x)$ by extending $S_n'(x)$. The property $S_n\subseteq{S_{n-1}}$ is obviously preserved during the extension.

\begin{figure}[h!]
  \includegraphics[width=340pt]{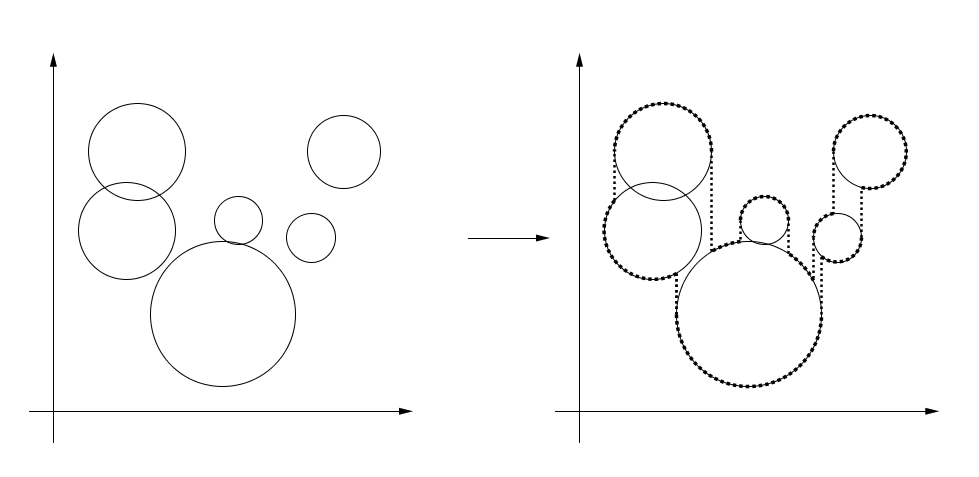}
  \caption{Extending an open set into an open strip}
  \label{szalag}
\end{figure}

\noindent To apply Proposition 3.1, we have to verify that $S=\cap_{n=1}^{\infty}S_n=G_0$. It is clear that $S$ contains $G_0$ since $S_n'$ contains every point of $G_0$ for all $n$. We have to show that $S$ has no other points. Proceeding towards a contradiction, let us assume that there exists a point $x\in{I}$ and $y\neq{f_0(x)}$ such that $(x,y)\in{S}$. We distinguish two cases.

\begin{enumerate}

\item The case $y>f_0(x)$. Since $(x,y)\in{S_n}$ for every $n$, the set $S_n'$ has a point $(x,z_n)$ above $(x,y)$. The sequence $(z_n)$ is obviously bounded, hence it has a limit point $z\geq{y}$. But $S_n'$ is formed by open balls whose centers are the points of $G_0\subseteq{T}$ and whose radii are not larger than $\frac{1}{n}$. Thus $(x,z)\in{T}$ as $T$ is closed. So $T$ has a point whose first coordinate is $x$ and whose second coordinate is larger than $f_0(x)=\max(T(x))$, a contradiction.

\item The case $y<f_0(x)$. By a similar argument to the previous one, we might notice that $S_n'$ has a point $(x,z_n)$ below $(x,y)$ for every $n$. Let $k\in\mathbb{N}$ satisfy $y<f_0(x)-\frac{1}{k}$. Then if $n\geq{2k}$, amongst the open balls forming $S_n'$ we might find a ball that intersects $v_x$ and for its center $(x_n,f_0(x_n))$ the inequality $f_0(x_n)<f_0(x)-\frac{1}{2k}$ holds. But by definition, it is impossible: this neighborhood must satisfy the overlapping condition, thus it cannot intersect $v_x$, a contradiction. Hence $f_0$ is a function of Baire class 1.

\end{enumerate} 

\noindent Using Lemma 3.1, we modify $f_0$ on a countable set $A$, so that the accumulation set of the new points above $A$ is $T$. We denote this altered function by $f$. Then it is a bounded Baire-2 function. Nevertheless, if we consider now the whole graph, $L_f=T$ remains true, since every point of the graph above $I\setminus{A}$ is in $T$. Therefore other accumulation points cannot occur. \end{proof}

In the following, we turn our attention to the not necessarily bounded Baire-2 functions. In this case the conditions are more complicated and the proof is a bit more difficult. However, we give a similar characterization.

\bigskip

We approach the problem by finding out some necessary conditions. During that process, we use only that $f:I\rightarrow\mathbb{R}$, as we did earlier in our previous theorem. It is easy to see that $T$ must be closed in this case, too. But it is not true at all that $L_f(x)=T(x)$ must be nonempty for every $x\in{I}$. For instance, let $f$ be the function that vanishes in $0$, and elsewhere its value is $\frac{1}{x}$. Then $L_f(0)$ is empty. Nevertheless, we may suspect that $T(x)$ cannot be empty in any set $C$. Our lemma is the following:

\begin{lemma} If $f:I\rightarrow\mathbb{R}$ and $C=\{x\in{I} : L_f(x)=\emptyset\}$, then $C$ is countable. \end{lemma}

\begin{proof} Proceeding towards a contradiction, let us assume that $C$ is uncountable. Put $C_n=\{x\in{C} : |f(x)|<n\}$ for every $n\in\mathbb{N}$. Then $C=\cup_{n=1}^{\infty}{C_n}$, and there exists an uncountable $C_n$. As a consequence, it contains one of its limit points, $c$. Thus there exists a sequence $(c_i)$ in $C_n$ ($c_i\neq{c}$) that converges to $c$. Since $(f(c_i))$ is bounded, it has a convergent subsequence, therefore $L_f(c)$ cannot be empty, a contradiction. \end{proof}

\bigskip

We state that these necessary conditions are also sufficient, namely:

\begin{theorem} Suppose $T\subseteq{I}\times\mathbb{R}$.  There is a Baire-2 function $f:I\rightarrow\mathbb{R}$ such that $L_f=T$ if and only if 
\begin{itemize}
\item $T$ is closed,
\item there is a countable $C\subseteq{I}$ such that $T(x)$ is nonempty for $x\in{I\setminus{C}}$.
\end{itemize}
\end{theorem}

\begin{proof}
The concept of the proof is similar to our proof given for the bounded case. We begin by the construction of a function $f_0$ and then we prove that it is a Baire-1 function. The desired function $f$ will be obtained by modifying $f_0$ on a countable set using Lemma 3.1.

\medskip

We start by observing that $C$ is a $G_\delta$ set. Suppose $c\in{C}$. Since $T$ is closed, it has a $B_{c,n}$ neighborhood for every $n\in\mathbb{N}$ such that for all $x\in{B_{c,n}}$ distinct from $c$, the absolute value of every element of $T(x)$ is larger than $n$. Otherwise $T(c)$ would not be empty. Then for a given $n$, the set $B_n=\cup_{c\in{C}}{B_{c,n}}$ is an open set containing $C$. On the other hand, clearly $\cap_{n=1}^{\infty}{B_n}=C$. Hence the set $C$ is $G_\delta$, as we wanted to show.

Now, we begin the construction of our function. The easier part is its definition on $C$. We consider an enumeration of the countable set $C=\{c_1, c_2, ...\}$ and we let $f_0(c_n)=n$ for every $n$. However, the definition of $f_0$ in $I\setminus{C}$ cannot be as straightforward as it was in our previous proof. Namely, it is possible that $T(x)$ has no maximum. Therefore we have to be more careful.

For every $n\in\mathbb{N}$, let

\begin{equation}\tag{4.2}\label{a} U_n=\{x\in{I} : \exists r\in{T(x)}, |r|\leq{n}\}. \end{equation}

\medskip

\noindent As $T$ is closed, it is easy to see that each $U_n$ is closed, too. It is also obvious that $U_{n}\subseteq{U_{n+1}}$ and $\cup_{n=1}^{\infty}{U_n}=I\setminus{C}$. Thus, for every $x\in{I\setminus{C}}$ there is a smallest $n_x$ such that $x\in{U_{n_x}}$. Using this property, we may define $f_0(x)$ as the largest element of $T(x)$, whose absolute value does not exceed $n_x$. We can do so since $T(x)$ is closed and it has such an element. The inequalities $n_x-1<|f_0(x)|\leq{n_x}$ are also true, as otherwise $x$ would be the element of $U_m$ for some $m<n_x$. (Or, if $n_x=1$, then $0=n_x-1\leq|f_0(x)|\leq{n_x}=1$.)

Now, we have defined $f_0$ on $I$. We would like to use Proposition 3.1 to show that $f_0$ is Baire-1. In order to do this, we construct the open strip $S_n$ for every $n$. First, we define the open set $S_n'$ constisting of some balls $B((x,f_0(x)),\varepsilon_{x,n})$. We select $\varepsilon_{x,n}$ so that $\varepsilon_{x,n}\leq\frac{1}{n}$ and $\varepsilon_{x,n}\leq\varepsilon_{x,n-1}$ for every $n\geq{2}$, as we did earlier. Nevertheless, as we defined $f_0$ differently in certain sets, our further conditions should be case-specific: we handle separately the case $x\in{C}$ and the case $x\in{I\setminus{C}}$.

\begin{enumerate}[(i)]

\item The case $x\in{C}$. It means that $x=c_k$ for some $k$. Let 
$$E_n=\cup_{x\in{C}}{B((x,f_0(x)),\varepsilon_{x,n})}\text{,}$$
and $F_n$ be its projection onto the $x$-axis, that is $F_n=\cup_{x\in{C}}{B(x,\varepsilon_{x,n})}$. Let us choose these neighborhoods such that $\cap_{n=1}^{\infty}F_n=C$. It is possib\-le since $C$ is a $G_\delta$ set. Furthermore, we also demand that $B(c_k,\varepsilon_{c_k,n})$ does not contain the points $c_1, ..., c_n$, with the exception of $c_k$. We remark that these conditions imply $\cap_{n=1}^{\infty}E_n$ equals the graph of $f_0|C$.

\item The case $x\in{I\setminus{C}}$. Let us make some remarks concerning this complementary set. Let $V_1=U_1$, and for $n\geq{2}$, let $V_n=U_{n}\setminus{U_{n-1}}$. Then the set $V_n$ is $F_\sigma$ for every $n$, as the difference of closed sets. Consequently, there exist closed sets $V_{n,i}$ for every $n$ and $i$ such that $V_n=\cup_{i=1}^{\infty}{V_{n,i}}$. We can take an enumeration $W_1, W_2, ...$ of the sets $V_{n,i}$. Let $x\in{V_k}$. We can suppose that the $\varepsilon_{x,n}$ are chosen so that $B(x,\varepsilon_{x,n})$ does not contain the points $c_1, c_2, ..., c_n$. Furthermore, we can suppose that $B(x,\varepsilon_{x,n})$ does not intersect the sets $W_1, W_2, ..., W_n$, except for those which contain $x$. Finally, we have a special overlapping condition, namely that $f_0(r)-f_0(x)<\frac{1}{n}$ for every $r\in{B(x,\varepsilon_{x,n})}\cap{V_k}$. One can prove that this condition can be satisfied as we proved it last time, in the bounded case. It is worth mentioning that if $f_0(x)<0$, then $(x,-(k-1))$ cannot be a limit point of a sequence of points in $G_0$ above $I\setminus{C}$. Since $T$ is closed, if such a sequence would exist, then $(x,-(k-1))\in{T}$. But it means that $x\in{U_{k-1}}$, hence $x\notin{V_k}$.

\end{enumerate}

\noindent Now the open set $S_n'$ is defined for each $n$. As in the bounded case, our next step is making strips of these open sets: let $S_n(x)=(\inf(S_n'(x)),\sup(S_n'(x)))$ for every $x\in{I}$. Set $\cap_{n=1}^{\infty}S_n=S$ and similarly $\cap_{n=1}^{\infty}S_n'=S'$. We are going to show that $S=G_0$. Since $G_0\subseteq{S}$ is obvious, we can focus on proving $S\subseteq{G_0}$, or equivalently, proving that $S$ has no point outside of $G_0$. We examine the relation of these sets independently for every $x\in{I}$: our goal is $S(x)\subseteq{G_0(x)}$. We distinguish the same cases which we distinguished during the construction of $S_n'(x)$:

\begin{enumerate}

\item The case $x\in{C}$, that is, $x=c_k$ for some $k$. Let us consider the set $S_n'(x)$. If $n\geq{k}$, amongst the open neighborhoods forming $S_n'$ there can be only one that intersects $v_x$: the neighborhood of $(x,f_0(x))$. Thus for sufficiently large $n$ the equality $S_n'(x)=S_n(x)$ holds, and $S_n'(x)$ contains only one open interval whose radius is $\frac{1}{n}$. Hence if $n$ converges to infinity, we find that the only element of $S(x)$ is $f_0(x)$. Therefore $S(x)\subseteq{G_0(x)}$.

\item The case $x\in{I\setminus{C}}$. It means $x\in{V_k}$ and $x\in{W_m}$ for some $k$ and $m$. Let us consider $S_n'(x)$. We would like to find out for which $r$ the open ball $B((r,f_0(r)),\varepsilon_{r,n})$ can intersect $v_x$. It is clear that for sufficiently large $n$ a neighborhood around a $(c_i,f_0(c_i))$ cannot do so as the intersection of these open balls are exactly the graph of $f_0|C$. Furthermore, if $n\geq{m}$, then the neighborhood chosen around $(r,f_0(r))$ can intersect $v_x$ if and only if $r\in{W_m}$. Indeed, we have chosen these neighborhoods such that they do not intersect $W_1, W_2, ..., W_n$, unless those which are containing $r$. Thus if $n$ is large enough, $v_x$ can be intersected by a certain $B((r,f_0(r)),\varepsilon_{r,n})$ only if $r\in{W_m}$. Only these places are relevant if we want to find out what $S(x)$ is. But how did we define $W_m$? It is a subset of $V_k$ thus the values of $f_0$ in $W_m$ are between $k-1$ and $k$. It is important to us that $f_0$ is bounded here, and $f_0(x)=\max(T_k(x))$ for each element of $W_m$, where $T_k=(I\times[-k,k])\cap{T}$, as in Lemma 3.1. Therefore, in the relevant places we defined $f_0$ as we would have done in Theorem 4.1, if we had regarded $T_k$ instead of $T$. Consequently, in this case one can conclude the proof of $S(x)\subseteq{G_0(x)}$ as it was done there.

\end{enumerate}

\noindent After these observations, the conclusion of the proof is clear. We use Lemma 3.1 as we did just before and alter the function on a countable set $A$, such that $L_f=T$ for the resulting function $f$. Then $f$ is obviously a Baire-2 function. \end{proof}

By proving this theorem we finished our characterization of accumulation points of Baire-2 functions. On the other hand, our proofs clarified that for any ordinal number $\alpha$ larger than $2$ the Baire-$\alpha$ functions are not interesting concerning our question. Namely, the accumulation set of the graph of a Baire-$\alpha$ function is also the accumulation set of a Baire-2 function. This fact explains why we examine only the Baire-1 and Baire-2 functions.

\section{Functions of Baire Class 1}

First, we focus again on the bounded case. Since Baire-1 functions are also Baire-2 functions, the conditions we found earlier recur in this case: $T$ should be compact and $T(x)$ should be nonempty, if $x\in{I}$. Nevertheless, it is clear, that these conditions are not sufficient. Namely, if $L_f=T$ and for a given $x$ the set $T(x)$ has multiple elements, then $f$ is discontinuous at $x$. But a Baire-1 function cannot have an arbitrary set of discontinuities: it must be a meager $F_\sigma$ set. Thus if $D=\{x: \#(T(x))>1\}$, then $D$ should be a meager $F_\sigma$ set. As we will see, these conditions suffice. However, before the statement of the actual theorem, let us notice that if we require $T$ to be closed, then it is redundant to require $D$ to be $F_\sigma$. Indeed, let $D_n=\{x: \text{diam}(T(x))\geq\frac{1}{n}\}$ for each $n\in\mathbb{N}$. Then it is easy to see that these sets are closed and their union is $D$. (Moreover, each $D_n$ is nowhere dense, otherwise some of them would contain an interval, and $D$ cannot do so.) Consequently, $D$ is an $F_\sigma$ set. Using this fact, our theorem is simply the following:

\begin{theorem} Suppose $T\subseteq{I}\times\mathbb{R}$.  There is a bounded Baire-1 function $f:I\rightarrow\mathbb{R}$ such that $L_f=T$ if and only if 
\begin{itemize}
\item $T$ is compact,
\item $T(x)$ is nonempty, if $x\in{I}$,
\item the set $D=\{x: \#(T(x))>1\}$ is meager.
\end{itemize}
\end{theorem}

\begin{proof}

Let us begin the proof by the construction of $f$. First, we use Lemma 3.1 to define $f$ on a countable set $A$ such that the accumulation set of the graph of $f|A$ coincides with $T$. We can suppose that $A$ is disjoint from $D$. Indeed, in any neighborhood of any point $x\in{I}$ there are infinitely many points of $I\setminus{D}$, since $D$ is meager. Thus we have defined $f$ on $A$. On the other hand, on $I\setminus{A}$ let us define $f$ as we did it in the bounded Baire-2 case: let $f(x)=\max(T(x))$. For this $f$, we have $L_f=T$, and obviously $f$ is bounded.

We would like to apply Proposition 3.1 to $f$. We use the usual method: we define the open set $S_n'$ for each $n$, which is the union of open balls around points of the graph with $\varepsilon_{x,n}$ radius, and then we extend these sets to open strips. The conditions concerning $\varepsilon_{x,n}$ will be case-specific, except for the usual size conditions.

\begin{enumerate}[(i)]

\item The case $x\in{A}=\{a_1, a_2, ...\}$. Then $x=a_k$ for some $k$. Our first condition on $\varepsilon_{x,n}$ is that $B(x,\varepsilon_{x,n})$ must not contain the points $a_1, a_2, ..., a_n$, except for $a_k$. The second condition is related to the overlapping of $D$. Since $D$ is a meager $F_\sigma$ set, we can choose $D_1, D_2, ...$ nowhere dense closed sets such that $D=\cup_{n=1}^{\infty}{D_n}$. Moreover, none of these sets contains $x$ since $x\in{A}$ and the sets $A$ and $D$ are disjoint. Therefore, the condition "$B(x,\varepsilon_{x,n})$ and $\cup_{i=1}^{n}{D_i}$ are disjoint" can also be satisfied.

\item The case $x\in{I\setminus{A}}$. First, in order to stay away from the set $A$, the open ball $B(x,\varepsilon_{x,n})$ must not contain the points $a_1, a_2, ..., a_n$. The second condition is identical to the overlapping condition of the bounded \mbox{Baire-2} case: if $r\in{B(x,\varepsilon_{x,n})}\setminus{A}$, then $f(r)-f(x)<\frac{1}{n}$.

\end{enumerate}

\noindent We have finished the construction of the open set $S_n'$, and now, we can extend it to obtain the open strip $S_n$ by taking the infimum and the supremum along each $v_x$. Our goal is to prove that the intersection $S$ of the sets $S_n$ is $G$. Of course, the challenging part is the verification of $S\subseteq{G}$. Let us consider $S(x)$ for each $x$. We separate three cases by the location of $x$:

\begin{enumerate}

\item The case $x\in{A}$, that is $x=a_k$. If $n\geq{k}$, then amongst the neighborhoods forming $S_n'$ there can be only one that intersects $v_x$, namely, the open ball centered at $(x,f(x))$. Therefore, $S_n(x)=S_n'(x)$, and

$$S_n(x)=(f(x)-\varepsilon_{x,n},f(x)+\varepsilon_{x,n})\subseteq\left(f(x)-\frac{1}{n},f(x)+\frac{1}{n}\right) \text{.}$$

\noindent This fact immediately implies that the only element of $S(x)$ is $f(x)$.

\item The case $x\in{D}$. It means that $x\in{D_k}$ for some $k$. Thus if $n\geq{k}$, the neighborhoods $B((a_k,f(a_k)),\varepsilon_{a_k,n})$ cannot intersect $v_x$. Therefore, if $n$ is sufficiently large, if we want to describe $S_n(x)$, we have to deal only with the points in $I\setminus{A}$. But above $I\setminus{A}$ we defined $f$ and the neighborhoods forming $S_n'$ as we defined $f_0$ and $S_n'$ in the proof of Theorem 4.1. Consequently, the proof given there for $S(x)=G_0(x)$ for any $x\in{I}$ works.

\item The case $x\in{I\setminus(A\cup{D})}$. Proceeding towards a contradiction, we assume that $S(x)$ has an element $y$ distinct from $f(x)$. Then $S_n'(x)$ has a point $z_n$ for each $n$ such that $|f(x)-z_n|\geq|f(x)-y|$. By definition, the set $G$ is bounded, thus it is obvious that there exists some $K\in\mathbb{R}$ such that for any $n$ and $x$, the $S_n'(x)$ has no element larger than $K$. It implies that the sequence $(z_n)$ is bounded. Therefore, it has a convergent subsequence whose limit is some $z\in\mathbb{R}$. For this limit $z$ the inequality $|f(x)-z|\geq|f(x)-y|$ also holds, thus $f(x)\neq{z}$. Since there is a point of $G$ whose distance from $(x,z_n)$ does not exceed $\frac{1}{n}$, the point $(x,z)$ is also an accumulation point of $G$, thus $(x,z)\in{L_f}$, a contradiction. Namely, for our $f$ the equation $L_f=T$ holds, however, the only element of $T(x)$ is $f(x)\neq{z}\in{L_f(x)}$.

\end{enumerate}

\noindent Therefore $S=G$, thus we can apply Proposition 3.1. Hence $f$ is a bounded Baire-1 function, such that $L_f=T$. \end{proof}

As we have characterized the bounded Baire-1 functions, now we might focus on the most challenging problem appearing in this paper: the characterization of the not necessarily bounded Baire-1 functions. However, as we will see, during the proof we will apply the same ideas. Following the usual scheme, we begin by thinking about necessary conditions concerning $T$.

\bigskip

The conditions we found during the examination of the general Baire-2 case obviously recur: $T$ is a closed set and $T(x)=\emptyset$ can hold only on a countable subset of $I$. As $T$ is closed, this subset is $G_\delta$. Of course we need more than these simple conditions. We have to pay attention to the fact that a Baire-1 function cannot have an arbitrary set of discontinuities: it must be a meager $F_\sigma$ set, and at points of continuity, $\#(L_f(x))=1$, thus $\#(T(x))=1$. However, we must be careful. In the bounded case, the property $\#(L_f(x))=1$ already guaranteed that $f$ is continuous at $x$, or $f$ has a removable discontinuity at $x$. But in this case, it is not true at all: for instance, if $f(x)=\frac{1}{2x-1}$ for $x>\frac{1}{2}$, and $f(x)=0$ for $x\leq\frac{1}{2}$, then although $L_f\left(\frac{1}{2}\right)=0$, it does not imply that $f$ is continuous at $\frac{1}{2}$ or it has a removable discontinuity there. Therefore, we must pay attention to the infinite limits. If we embed $T$ into $I\times\overline{\mathbb{R}}$ and take its closure $\overline{T}$, then we have to demand that this $\overline{T}$ can intersect the extended vertical lines in multiple points only above a meager $F_\sigma$ set. However, the additional $F_\sigma$ condition is unnecessary since we supposed that $T$ is closed. Indeed, if $D_n=\{x: \text{diam}(\overline{T}(x))\geq\frac{1}{n}\}$, then these sets are nowhere dense closed sets and their union is $D$, hence $D$ is $F_\sigma$.

\medskip

If we collect all of these remarks, we gain a more complicated system of conditions than the ones in the previous cases. We show that it is sufficient.

\medskip

\begin{theorem} Suppose $T\subseteq{I}\times\mathbb{R}$.  There is a Baire-1 function $f:I\rightarrow\mathbb{R}$ such that $L_f=T$ if and only if 
\begin{itemize}
\item $T$ is closed,
\item there is a countable $C\subseteq{I}$, such that $T(x)$ is nonempty for $x\in{I\setminus{C}}$,
\item the set $D=\{x: \#(\overline{T}(x))>1\}$ is meager.
\end{itemize}
\end{theorem}

\begin{proof}

We define $f$ on a countable set $A$, such that the accumulation set of the graph of $f$ restricted to $A$ equals $T$. We do so using the method given in Lemma 3.1. It is easy to see that we can construct such a set $A$ disjoint from $C$ and $D$.

Now let us focus on $I\setminus{A}$. We define $f$ on this set as we defined $f_0$ in the proof of Theorem 4.2. First, if $C=\{c_1,c_2,...\}$, then $f(c_n)=n$ for each $n\in\mathbb{N}$. Besides that we also define $U_n$ as we did it in \eqref{a}. These are closed sets in this case, too, though not necessarily disjoint from $A$. At places which are not in $A$ let us define $f$ as we defined $f_0$ after \eqref{a}: if $x\in{U_n}$, let $f_0(x)$ be the largest element of $T(x)$ which has absolute value not exceeding $n$. Now we are ready with the construction of $f$ and $L_f=T$ clearly holds: if we consider only the points of the graph above $A$, it is true by definition, furthermore, sequences containing infinitely many points of the graph above $C$ cannot converge, and points of the graph above $I\setminus(A\cup{C})$ are in $T$. Thus every accumulation point of $G$ is also the accumulation point of the graph of $f|A$, and the set of these accumulation points is $T$. (We note that $C$ might intersect $D$, a concern that we will address later.)

We would like to apply Proposition 3.1 to $f$ by giving the open sets $S_n'$ formed by neighborhoods of points of $G$ and extending them to open strips. Again, we separate some cases. We also use our familiar notation: $A=\{a_1,a_2,...\}$, $C=\{c_1,c_2,...\}$, and $D=\cup_{n=1}^{\infty}{D_n}$, where $D_n$ is a nowhere dense, closed set for each $n$.

\begin{enumerate}[(i)]

\item The case $x\in{C}$, $x=c_k$. Here, we define our neighborhoods with $\varepsilon_{x,n}$ radius quite comfortably, namely, we can define the sets $E_n$ and $F_n$ as we did it in (i) of the proof of Theorem 4.2 and repeat the conditions used there. Hence we can choose these open balls such that $\cap_{n=1}^{\infty}F_n=C$, and $B(c_k,\varepsilon_{c_k,n})$ does not contain the points $c_1, ..., c_n$, with the exception of $c_k$. We also require that this neighborhood is disjoint from $\{a_1, a_2, ..., a_n\}$. We remark that these conditions imply $\cap_{n=1}^{\infty}E_n$ equals the graph of $f_0|C$.

\item The case $x\in{A}$. We evoke the conditions of (i) of the proof of Theorem 5.1. Namely, $B(x,\varepsilon_{x,n})$ does not intersect the closed sets $D_1, D_2, ..., D_n$, and it does not contain $a_1, a_2, ..., a_n$, with the exception of $x$. Furthermore we give the following additional condition: these neighborhoods have to stay away from $C$, thus they must not contain $c_1,c_2,...,c_n$.

\item The case $x\in{I\setminus(A\cup{C}})$. We evoke the condition system of (ii) of the proof of Theorem 4.2. We define the sets $V_n$ and $W_n$ as we did there: $V_1=U_1$, and $V_n=U_{n}\setminus{U_{n-1}}$ for $n\geq{2}$. Then any set $V_n$ is $F_\sigma$. Let $W_1, W_2, ...$ be an enumeration of the closed sets forming them. Now if $x\in{V_k}$, we require $B(x,\varepsilon_{x,n})$ to be disjoint from $c_1, c_2, ..., c_n$, and also disjoint from the sets $W_1, W_2, ..., W_n$, except for those containing $x$. Furthermore, of course, we give an overlapping condition: $f_0(r)-f_0(x)<\frac{1}{n}$ for each $r\in{B(x,\varepsilon_{x,n})}\cap{V_k}$. These are exactly the conditions we used in (ii) of the proof of Theorem 4.2. The only additional condition is the following: $B(x,\varepsilon_{x,n})$ must not contain the points $a_1,a_2,...,a_n$.

\end{enumerate}

\noindent Thus we have constructed the open set $S_n'$ for each $n$. We extend it in the usual way to form the open strip $S_n$. Our goal is to verify that their intersection $S$ equals $G$. The challenging part is to show that $S$ contains no points distinct from $G$. Let us consider $S(x)$ and $S'(x)$ for each $x$. We separate four cases by the location of $x$:

\begin{enumerate}

\item The case $x\in{C}$, $x=c_k$. This is obvious: if $n\geq{k}$, the only chosen neighborhood that intersects $v_x$ amongst the ones forming $S_n'(x)$ is the neighborhood of $(x,f(x))$, and thus $S_n'(x)=S_n(x)$. Therefore, $S_n(x)$ is an interval whose diameter does not exceed $\frac{2}{n}$ and contains $f(x)$. Thus the only element of $S(x)$ is $f(x)$, as we wanted to show.

\item The case $x\in{A}$, $x=a_k$. We can simply repeat our previous argument: for sufficiently large $n$, there is only one chosen neighborhood that intersects $v_x$, and since the diameters of these neighborhoods converge to $0$, the only element of $S(x)$ is $f(x)$.

\item The case $x\in{D\setminus{C}}$. It means $x\in{D_k}$ for some $k\in\mathbb{N}$. Now, if $n\geq{k}$, the  neighborhood $B((x',f(x')),\varepsilon_{x',n})$ for $x'\in{A}$ cannot intersect $v_x$. It is also true that for sufficiently large $n$, the neighborhood $B((x',f(x')),\varepsilon_{x',n})$ for $x'\in{C}$ cannot intersect $v_x$, since these neighborhoods are nested and their intersection is the graph of $f|C$. Hence it is enough to consider the graph of $f$ above $I\setminus(A\cup{C})$. At these places we defined $f$ and the open balls forming $S_n'$ as we defined $f_0$ and the open balls forming $S_n'$ during the proof of Theorem 4.2. Consequently, case (2) of the proof of Theorem 4.2 can be used to prove $S(x)=G(x)$.

\item The case $x\in{I\setminus(A\cup{C}\cup{D})}$. Proceeding towards a contradiction, let us suppose that $S(x)$ contains some $y\in\mathbb{R}$, where $y\neq{f(x)}$. It means that for every $n$ we can choose a point $z_n$ in $S_n'(x)$, such that $|f(x)-z_n|\geq|f(x)-y|$. Since $z_n\in{S_n'(x)}$, the point $(x,z_n)$ is in one of the open balls forming $S_n'$. Here, if $n$ is sufficiently large, then this ball is centered at a point of the graph above $I\setminus{C}$. Indeed, if $n$ is large enough, the neighborhoods around points of the graph above $C$ cannot intersect $v_x$ by definition.
\noindent Now, the sequence $(z_n)$ has a limit point $z$ in $\overline{\mathbb{R}}$. Obviously, for this $z$ the inequality $|f(x)-z|\geq|f(x)-y|$ also holds, thus $f(x)\neq{z}$. However, if $n$ is sufficiently large, there is a point of the graph not above $C$ whose distance from $(x,z_n)$ does not exceed $\frac{1}{n}$. Consequently, there is a sequence $(p_n)$ of points of the graph above $I\setminus{C}$ such that $(p_n)$ converges to $(x,z)$. Without loss of generality, we might assume that the elements of this sequence are all distinct. Since these points are not above $C$, they are above $A$ or they are also elements of $T$. Nevertheless, if $n$ is sufficiently large, for any given $\varepsilon>0$, a point $p_n$ that is above $A$ cannot be farther than $\varepsilon$ from a point of $T$, as we noted in Remark 3.1. This fact immediately implies $(x,z)\in\overline{T}$, a contradiction, since the only element of $\overline{T}(x)$ is $f(x)$ by our assumptions.

\end{enumerate}

\noindent Hence $S=G$, therefore we might apply Proposition 3.1. Thus $f$ is a Baire-1 function satisfying $L_f=T$. \end{proof}

\section{Concluding Remark}

Before the end of this paper, we would like to point out something in connection with our theorems about the not necessarily bounded functions. Namely, amongst the conditions of the last theorem there was one condition about $\overline{T}$. However, $\overline{T}=\overline{L_f}$ does not necessarily hold for the function we constructed.

\bigskip 

For instance let $T$ be the following closed set: let $C=\{\frac{1}{n} : n\in{\mathbb{N}}\}\cup\{0\}$, $c_1=0$, and for $n\geq{2}$, let $c_n=\frac{1}{n-1}$. For each point $x$ in $I\setminus{C}$ let $T(x)=\{-\frac{1}{d(x,C)}\}$, where $d(x,C)$ is the distance of $x$ from $C$. Then it is easy to see that this set $T$ satisfies the conditions of Theorem 5.2 with regards to the not necessarily bounded Baire-1 functions. It is also true, that $\overline{T}(0)=\{-\infty\}$. Now, let us consider $f$, specifically $\overline{L_f}(0)$. We recall that in our construction $f(c_n)=n$. It implies $\overline{L_f}(0)=\{-\infty,+\infty\}$. It means that although $L_f(0)=T(0)=\emptyset$, $\overline{T}(0)\neq\overline{L_f}(0)$. Thus the sets we examined earlier are equal, but these extended sets are not.

\medskip

This example raises two new questions: if we regard our theorems about the not necessarily bounded Baire-1 and Baire-2 functions and we do not change the conditions, is it possible to construct a function $f$ in each of these cases that satisfies $L_f=T$ and $\overline{L_f}=\overline{T}$ simultaneously? However, we might answer these questions easily:

\begin{prop} Suppose $T\subseteq{I}\times\mathbb{R}$. 
\begin{itemize}
\item If there exists a Baire-2 function satisfying $L_f=T$, then it can be chosen such that $\overline{L_f}=\overline{T}$ also holds.
\item If there exists a Baire-1 function satisfying $L_f=T$, then it can be chosen such that $\overline{L_f}=\overline{T}$ also holds.
\end{itemize}
\end{prop}
 
\begin{proof} We will appropriately modify the functions we have constructed in the proofs of Theorem 4.2 and Theorem 5.2. It is clear that for those functions $\overline{T}\subseteq{\overline{L_f}}$ holds. Indeed, for any point $t\in{T}$ there are points of $G$ arbitrarily close to $t$. Thus if we consider a point $(x,\infty)$ of $\overline{T}$, then it is also an accumulation point of $G$. Hence if $\overline{L_f}\neq\overline{T}$, then $\overline{T}$ is a proper subset of $\overline{L_f}$.

For those functions it is also clear that if $\overline{L_f}$ has a point $p$ which is not in $\overline{T}$, then it is an accumulation point of the graph of $f|C$. Namely, if we take a sequence $(p_n)$ in $G$ which converges in $I\times\overline{\mathbb{R}}$ and contains only finitely many points of $G$ above $C$, then after a while every term of this sequence is above $A$ or in $T$. The terms above $A$ will get arbitrarily close to a point of $T$ if $n$ is sufficiently large. Thus if we have a point in $\overline{L_f}$ which is a limit point of such a sequence, then it is also a point of $\overline{T}$. Hence if $\overline{L_f}$ has a point outside $\overline{T}$, then there exists a sequence in the graph of $f|C$ converging to this point.

It is a problem we can easily handle in both cases by modifying $f$ on $C$: if $C=\{c_1,c_2,...\}$, then let $|f(c_n)|=n$. The sign is determined by whether $\overline{T}$ contains $(c_n,+\infty)$ or $(c_n,-\infty)$. If both of them occurs, then let $f(c_n)=n$. If we define the function $f$ on $C$ this way, then $L_f$ clearly does not change, the equality $L_f=T$ still holds. Indeed, if a sequence of points of $G$ above $C$ converges to a point in $I\times\overline{\mathbb{R}}$, then the second coordinate of this point is $+\infty$ or $-\infty$. By symmetry, we can consider the $+\infty$ case. For a subsequence $\left(c_{n_k}\right)$ the sequence $\left(c_{n_k},f(c_{n_k})\right)$ converges to some $(x,+\infty)\in{I\times\mathbb{R}}$. We can suppose that all the numbers $f(c_{n_k})$ are positive. Then by definition, in the $\frac{1}{n_k}$ neighborhood of $c_{n_k}$  we might choose a point $a_k$ such that $T(a_k)$ has an element larger than $n_k$. We denote this element of $T$ by $t_k$. Now it is clear that the sequence $(t_k)$ is in $T$ and it also converges to $(x,+\infty)$. Hence all the elements of $\overline{L_f}$ are in $\overline{T}$, too. Thus we constructed a function of the corresponding Baire class satisfying $L_f=T$ and $\overline{L_f}=\overline{T}$ simultaneously. \end{proof}

\subsection*{Acknowledgements} 

I am grateful to Zolt\'an Buczolich for his valuable remarks and for the time he spent with reading this paper several times before submission which helped to improve the quality of the exposition. I am also grateful to P\'eter Maga for his useful comments. Finally, I would like to thank the anonymous referee the thorough report and the linguistic suggestions.

\end{document}